\documentclass[12pt]{amsart}
\usepackage{amsmath, amssymb, amsthm, latexsym}
\usepackage{mathrsfs}
\usepackage{amssymb}
\usepackage{latexsym}
\usepackage[center]{caption}
\usepackage[OT2,T1]{fontenc}
\DeclareSymbolFont{cyrletters}{OT2}{wncyr}{m}{n}
\DeclareMathSymbol{\Sha}{\mathalpha}{cyrletters}{"58}
\usepackage{tikz}
\newcounter{braid}
\newcounter{strands}
\pgfkeyssetvalue{/tikz/braid height}{1cm}
\pgfkeyssetvalue{/tikz/braid width}{1cm}
\pgfkeyssetvalue{/tikz/braid start}{(0,0)}
\pgfkeyssetvalue{/tikz/braid colour}{black}
\pgfkeys{/tikz/strands/.code={\setcounter{strands}{#1}}}

\makeatletter
\def\cross{%
  \@ifnextchar^{\message{Got sup}\cross@sup}{\cross@sub}}

\def\cross@sup^#1_#2{\render@cross{#2}{#1}}

\def\cross@sub_#1{\@ifnextchar^{\cross@@sub{#1}}{\render@cross{#1}{1}}}

\def\cross@@sub#1^#2{\render@cross{#1}{#2}}

\def\render@cross#1#2{
  \def\strand{#1}
  \def\crossing{#2}
  \pgfmathsetmacro{\cross@y}{-\value{braid}*\braid@h}
  \pgfmathtruncatemacro{\nextstrand}{#1+1}
  \foreach \thread in {1,...,\value{strands}}
  {
    \pgfmathsetmacro{\strand@x}{\thread * \braid@w}
    \ifnum\thread=\strand
    \pgfmathsetmacro{\over@x}{\strand * \braid@w + .5*(1 - \crossing) * \braid@w}
    \pgfmathsetmacro{\under@x}{\strand * \braid@w + .5*(1 + \crossing) * \braid@w}
    \draw[braid] \pgfkeysvalueof{/tikz/braid start} +(\under@x pt,\cross@y pt) to[out=-90,in=90] +(\over@x pt,\cross@y pt -\braid@h);
    \draw[braid] \pgfkeysvalueof{/tikz/braid start} +(\over@x pt,\cross@y pt) to[out=-90,in=90] +(\under@x pt,\cross@y pt -\braid@h);
    \else
    \ifnum\thread=\nextstrand
    \else
     \draw[braid] \pgfkeysvalueof{/tikz/braid start} ++(\strand@x pt,\cross@y pt) -- ++(0,-\braid@h);
    \fi
   \fi
  }
  \stepcounter{braid}
}

\tikzset{braid/.style={double=\pgfkeysvalueof{/tikz/braid colour},double distance=1pt,line width=2pt,white}}

\newcommand{\braid}[2][]{%
  \begingroup
  \pgfkeys{/tikz/strands=2}
  \tikzset{#1}
  \pgfkeysgetvalue{/tikz/braid width}{\braid@w}
  \pgfkeysgetvalue{/tikz/braid height}{\braid@h}
  \setcounter{braid}{0}
  \let\sigma=\cross
  #2
  \endgroup
}
\makeatother

\input xypic

\swapnumbers
\newtheorem{theorem}{Theorem}

\newtheorem{lemma}[theorem]{Lemma}
\newtheorem{construction}[theorem]{Theorem-Construction}


\def\Z{\mathbb{Z}}

\def\C{\mathbb{C}}

\def\Q{\mathbb{Q}}

\def\C{\mathbb{C}}

\def\N{\mathbb{N}}

\def\F{\mathbb{F}}

\def\G{\mathbb{G}}

\def\qed{\hfill$\square$\medskip}

\def\Zpk{\mathbb{Z}/p^{k}}
\def\Zpk1{\mathbb{Z}/p^{k-1}}

\newcommand{\rref}[1]{(\ref{#1})}

\newcommand{\beg}[2]{\begin{equation}\label{#1}#2\end{equation}}
\def\r{\rightarrow}

\def\F{\mathbb{F}}

\def\sl2{\widetilde{SL_{2}(\Z)}}

\title[Equivariant complex-oriented spectra]{Equivariant formal group laws and 
complex-oriented spectra over primary cyclic groups:
Elliptic curves, Barsotti-Tate groups, and other examples}
\author{Po Hu, Igor Kriz and Petr Somberg}
\thanks{The authors acknowledge support by grants GA\,CR P201/12/G028 and GA\,CR 19-28628X.
Kriz also acknowledges the support of a Simons Collaboration Grant.}

\begin{document}
\maketitle

\begin{abstract}
We explicitly construct
and investigate a number of examples
of $\Z/p^r$-equivariant formal group laws and complex-oriented spectra,
including those coming from elliptic curves and $p$-divisible groups,
as well as some other related examples.
\end{abstract}

\section{Introduction}

The concept of complex-oriented generalized cohomology theory, i.e. a commutative ring spectrum $E$ for
which all complex bundles (equivalently, the universal line bundle on $\C P^\infty$) are oriented, has 
played a central role in stable homotopy theory. The reason is that $E^*(\C P^\infty)$ then has 
a structure of a ($1$-dimensional commutative) formal group law over the coefficient ring $E^*$, 
and that generically, via this connection, formal group laws do tend to appear in homotopy theory. 
The formal group law (FGL) of ordinary (singular) cohomology is additive, the FGL of K-theory is
multiplicative.
In 1969, D.Quillen \cite{q} proved that the formal group law of complex cobordism $MU$ is the universal formal group law,
thus identifying the coefficients $MU_*$ with the Lazard ring. To this present day, the best 
definition of ``elliptic cohomology" is as a complex-oriented cohomology theory whose formal group
law comes from an elliptic curve. Complex-oriented cohomology theories associated with Lubin-Tate laws have been
constructed by Landweber \cite{land}, and given $E_\infty$-ring structures by Goerss and Hopkins \cite{gh}.

\vspace{3mm}

It is a very natural question whether this spectacular story has an $G$-equivariant analogue. The answer is 
very likely positive, at least in the case of a finite abelian group $G$. Yet, at the present time, only 
a small part of the story has been developed, as we shall explain below. The aim of this paper is to advance
this work in progress a little further. 

\vspace{3mm}

A systematic treatment of equivariant spectra can be found in \cite{lms}. The essential point is that 
in order to have Spanier-Whitehead duality in the equivariant case, cohomology theories must be indexed
not by integers, but by virtual representations of the group $G$ (because a finite-dimensional $G$-CW-complex
can be embedded into a finite-dimensional representation, but of course not, in general, a trivial one). This
corresponds to indexing a spectrum by finite-dimensional subrepresentations of a {\em complete universe} 
$\mathcal{U}$, which, for our purposes, can be thought of as just a sum of countably many copies of $\C[G]$.

\vspace{3mm}

A complex-oriented $G$-equivariant spectrum $E$ then is one for which every $G$-equivariant complex
bundle is $E$-oriented. Again, it suffices to verify this for the universal complex line bundle on
$\C P^\infty_G$, which is the space of complex lines through the origin in $\mathcal{U}$.
As a matter of fact, because of failure of transversality,
geometric cobordism of $G$-equivariant compact smooth manifolds is not
represented by a complex-oriented equivariant spectrum.
However, a complex-oriented 
$G$-equivariant analog $MU_G$ of the complex Thom spectrum can be formed, and represents
a certain stabilization of equivariant complex cobordism, corresponding precisely to ``fixing" the 
transversality problem. The coefficients of $MU_G$ for $G=\Z/p$ with $p$ prime were calculated in 
\cite{zp}. The calculation for a finite abelian group was made in \cite{ak}. The answer is substantially more
complicated than the Lazard ring, and it is not apparent at first glance what it means algebraically.

\vspace{3mm}

The concept of a $G$-equivariant formal group for an abelian group $G$ was introduced in \cite{cgk}
to model the structure present on $E^*(\C P^\infty_G)$ for a $G$-equivariant complex-oriented spectrum $E$
and a finite abelian group $G$. The concept is also much more convoluted: for $G$ a finite abelian group,
an {\em equivariant FGL} over a commutative ring $A$ consists of a formal $A$-coalgebra $B$
together with a coalgebra map 
\beg{edefm}{B\r A[G^*]^\vee} 
(where $G^*$ is the character group of $G$ and $?^\vee$ 
denotes the $A$-linear dual), such that the defining ideal $I$ of $B$ is generated by $\prod_{L\in G^*} x_L$
where $x_L\in B$ are elements such that $B/(x_L)\cong A$, and $x_L$ are related to one another
by the comultiplication \rref{edefm} ``in the obvious way" (i.e. taking the comultiplication on $X_L$,
projecting one coordinate via \rref{edefm}, and taking the value at 
another character $M$, gives $X_{L(M^{-1})}$.

\vspace{3mm}

It is easy to see the universal $G$-equivariant FGL is represented by a commutative ring $L_G$, and from
the complex orientation, one obtains a map 
\beg{emugor}{L_G\r (MU_G)_*.}
It is therefore a natural question whether \rref{emugor} is an isomorphism. This was recently answered
in the affirmative by Hausmann \cite{hausmann}. Hanke and Wiemeler \cite{hw} previously
proved it in the case $G=\Z/2$, using an explicit presentation
of $(MU_{\Z/2})_*$ in terms of generators and defining relations, due to Strickland \cite{strick}. Kriz
and Lu \cite{klu} also recently obtained an alternative proof for a general abelian compact Lie group, using
other methods.

\vspace{3mm}
Strickland \cite{stricmult} defined the equivariant formal group law associated with an elliptic curve,
and the concept of equivariant elliptic cohomology. It should be noted that $S^1$-equivariant elliptic cohomology
is a concept desirable also from the purely non-equivariant point of view: when defining a non-equivariant 
elliptic spectrum
as a complex-oriented spectrum being associated with the FGL of an elliptic curve, then $E^*(\C P^\infty)$
should be a completion of the ring of coefficients of $\mathcal{E}_{S^1}$. Such a spectrum, as far as we know, has
in the published literature only been constructed rationally by Greenlees \cite{grat,ag}. (The rational
case is non-trivial over $G=S^1$.) One can also ask about equivariant lifts (or, in other words, deformations), 
of Lubin-Tate formal group laws.

\vspace{3mm}
The goal of this paper is to investigate these questions  by looking at concrete
examples in the case of primary cyclic groups,
i.e. the groups $G=\Z/(p^r)$, and observing some of the properties of these constructions. In the elliptic case, we construct, completely explicitly, using equivariant
isotropy separation, a $\Z/p$-equivariant elliptic spectrum on one version of the ``universal" elliptic curve for
$p>3$. One then observes that this construction of equivariant elliptic cohomology only
depends on the $p$-divisible group associated with the elliptic curve. It is therefore
natural to look for $\Z/p^r$-equivariant deformations
of Lubin-Tate laws from $p$-divisible groups. We do describe such a construction, starting
with the underlying
{\em universal deformation rings $R$ of $p$-divisible groups over perfect fields with one-dimensional
formal parts}. More precisely,
the equivariant FGL is over the coefficient ring of the $p^r$-torsion subgroup of the $p$-divisible group on $R$.
We also construct complex-oriented $\Z/p^r$-equivariant cohomology theories (which are
$E_\infty$-ring spectra) associated to these equivariant FGL's. We call them {\em Barsotti-Tate cohomology 
theories}. Because we started with $p$-divisible and not formal groups, the resulting theory is non-trivial, i.e. 
is not just a Borel cohomology theory. This can be observed in the elliptic case, where at primes of
ordinary good reduction, we obtain the equivariant multiplicative formal group law. The story is essentially
the same for $G=\Z/(p^r)$ as for $G=\Z/p$, but the algebra is more complicated,
and more involved isotropy separation techniques are required. For
that reason, the body of this paper is written for $G=\Z/p$, and the generalization to $\Z/(p^r)$ is 
treated in Section \ref{szpr}. 

\vspace{3mm}
In Section \ref{sfurth}, we then discuss these constructions further. Namely, one observes that
all the $\Z/p^r$-equivariant formal group laws and spectra are completely separated by isotropy,
i.e. that we are dealing with products of rings (or ring spectra). We observe that they are still 
non-trivial in the sense that they give non-trivial deformations of the underlying non-equivariant
formal group law, but this begs the question whether Zariski-connected examples exist. 
We observe that such examples (not coming from $p$-divisible groups) are easy to construct, but
it is harder to construct them on Noetherian rings. Possible constructions use 
the complete calculation of the coefficients of $\Z/p^r$-equivariant stable cobordism due to Hu \cite{hu},
which is very recent work. Because of this, in limit our discussion in the present paper to the case $p^r=2$,
where the cobordism calculation is ``classical," due to Strickland \cite{strick}.

\vspace{3mm}
Finally, we should mention that there have been quite a few other developments in the direction
of the program we outlined, both before and after our paper was written, even though the overlap 
of the present paper with the existing work is surprisingly small. Jacob Lurie has written a series of
papers \cite{lurie,lurie1} on elliptic cohomology, in which equivariant elliptic cohomology is also
discussed. As far as we can tell,  those papers do not discuss ``genuine" equivariant
cohomology theories in the sense considered here. Since the first version of the present paper was first
submitted for publication, however, a paper of Gepner and Meier \cite{gepnermayer}, which set out to remedy
this in the general case, appeared on the arXiv. One should also mention another theorem of J. Lurie
(see Goerss \cite{goerss}), stating that a map from a Deligne-Mumford stack into the
stack of formal groups is realizable in the category of $E_\infty$-ring
spectra if it factors through the stack of $p$-divisible groups. This theorem played a crucial
role for example in the construction of topological automorphic forms by Behrens and Lawson \cite{bl}.
The role of $p$-divisible groups in that context is somewhat different; in our case, the 
key point (which also favors the case of $\Z/p^r$ as ``special") is that these groups are 
subgroups of $S^1$, which, from the point of view of global homotopy theory, underlies the
structure of formal group laws. 

On the subject of global homotopy theory, we cannot, of course, neglect to mention again Hausmann's spectacular
proof of the Greenlees conjecture \cite{hausmann}, stating that for any abelian compact Lie
group $G$, $MU^G_*$ represents the universal $G$-equivariant formal group law. This generalizes
a previous result of Hanke and Wiemeler \cite{hw}, which established this for $G=\Z/2$.
Hausmann used global homotopy theory for his proof, and a different proof of universality of 
$MU^G_*$ was since obtained by Kriz and Lu \cite{klu}, which also gives a presentation
of $MU^G_*$ as a ``deformation" of the Lazard ring.

The development calculationally closest to the present paper is the complete computation by Hu \cite{hu}
of the coefficient of $\Z/p^r$-equivariant stable cobordism. This computation, in fact,
allows some of the results of the present paper to be further generalized, but this is rather
technical, and for this reason will be left for future work.

\vspace{3mm}
The present paper is organized as follows. In Section \ref{sell}, we prove some algebraic preliminaries. 
In Section \ref{sequivconst}, we construct $\Z/p$-equivariant elliptic spectra. We further investigate
their formal groups in Section \ref{sco}, and prove that they coincide with Strickland's equivariant elliptic
FGL's. In Section \ref{spdiv}, we construct equivariant formal group laws from $p$-divisible
groups with one-dimensional formal parts. In Section \ref{sbt}, we construct $\Z/p$-equivariant
Barsotti-Tate cohomology. 
In Section \ref{szpr}, we carry out the generalizations of all of the constructions of this paper to the case
of the group $\Z/(p^r)$. In Section \ref{sfurth},
we further discuss these constructions, and give constructions of equivariant deformations 
of multiplicative and Lubin-Tate formal group laws, and corresponding complex-oriented
spectra, which are Zariski-connected. Section \ref{sapp} is an Appendix containing a brief comparison of the equivariant 
isotropy separation techniques we use in the present paper with the methods of \cite{ak}, in view of
the recent results \cite{hu}. 

\vspace{5mm}

\section{On the $p$-torsion of universal elliptic curves}\label{sell}

Let $R=\Z[1/6][g_2,g_3][\Delta^{-1}]$ where $\Delta=g_2^3-27g_3^2$. Over $Spec(R)$, we have
the elliptic curve
$$E=Proj(R[X,Y,Z]/(Y^2Z-4X^3+g_2XZ^2+g_3Z^3)).$$
Denote by ${}_pE\subset E$ the reduced closed group subscheme of $p$-torsion points. Consider the affine
open subset $U=U_Y$. Thus, setting $x=X/Y$, $u=Z/Y$, we have
$$U=Spec(R_0)$$
where
\beg{er0}{R_0=R[x,u]/(u-4x^3+g_2xu^2+g_3u^3).}
The open set $U$ is the complement of the closed subset ${}_2E\smallsetminus {}_1E$
in $E$. (Here, ${}_1E$ is, of course, the unit.)
As Strickland \cite{stricmult}, Chapter 7, 
points out, for $p>2$ prime, 
$${}_pE\subset U$$
(because, denoting $[n]P$ the $n$-multiple, $[2]P=[p]P=0$ implies $P=0$). Thus, ${}_pE$ is affine
and there exists an ideal $I\subset R_0$
with
\beg{eforc}{\mathcal{O}_{{}_pE}=R_0/I.}
Using \rref{er0}, we may express $u$ recursively as a power series in $x$ with coefficients in $R$, in
terms of which the group law of $E$ becomes a formal group law on $R$. Consequently, the formal
completion of $E$ at ${}_1E$ induces 
a canonical homomorphism of rings
\beg{efc}{\mathcal{O}_{{}_pE}\r R[[x]]/[p]x}
(where here $[n]x$ denotes formal multiple).
Strickland \cite{stricmult} concludes that there is a $\Z/p$-equivariant formal group law on $\mathcal{O}_{{}_{p}E}$.
From now on, we will assume that $p>3$ is prime.

\begin{theorem}\label{t1}
The ring $\mathcal{O}_{{}_pE}$ is finite and flat over $R$. Furthermore, the square (induced by \rref{eforc})
\beg{epullback}{\diagram
\mathcal{O}_{{}_pE}\rto\dto & x^{-1} \mathcal{O}_{{}_pE}\dto\\
R[[x]]/[p]x\rto & x^{-1} R[[x]]/[p]x
\enddiagram
}
is strictly Cartesian. (This means that it is a pullback, and the bottom row and right column are jointly onto.)
\end{theorem}

\vspace{3mm}
{\bf Comment:} Note that \rref{epullback} is the Tate square in local homology
of the commutative ring $\mathcal{O}_{{}_pE}$
in the sense of Greenlees \cite{greenleestate} with respect to the ideal $(x)$. This square has been also used
in \cite{strick,hu} in connection with calculations of coefficient rings of equivariant formal group laws.

\begin{proof}
Denote by $\mathcal{O}^\prime_{{}_pE}$ the pullback of the three terms of \rref{epullback} other than 
the upper left corner. Then we have a canonical homomorphism of rings
\beg{ech}{\mathcal{O}_{{}_pE}\r \mathcal{O}^\prime_{{}_pE}.
}
Our task is to show that it is an isomorphism. First notice that the right hand column and bottom row of 
\rref{epullback} are jointly onto. This is because $x^{-1}R[[x]]$ is generated by $R[[x]]$ and $R[x^{-1}]$.
Thus, we have a short exact sequence
\beg{ech1}{
0\r \mathcal{O}^\prime_{{}_pE}\r x^{-1}\mathcal{O}_{{}_pE}\oplus
R[[x]]/[p]x\r x^{-1}R[[x]]/[p]x\r 0.
}
Now notice that \rref{ech} is obviously an isomorphism after inverting $x$. Next, we will show that it is
an isomorphism modulo $x$. To this end, map \rref{ech1} to itself by $x$, and apply the snake lemma.
The kernel and cokernel of the right column are $0$, and the cokernel of the middle column is
$R$. Thus, we obtain
$$\mathcal{O}^\prime_{{}_pE}/(x)=R,$$
as claimed. Since \rref{ech} is an isomorphism at $x$ and away from $x$, by Zariski's main theorem, it is
locally finite. (Note that by a similar argument, $\mathcal{O}_{{}_{p}E}$ is locally finite over $Spec(R)$.)
Now to show that \rref{ech} is an isomorphism, it suffices to prove that it is an isomorphism when localized
at a maximal ideal $m$ of $R$ (using the Nullstellensatz). Note that $m$ will contain a prime number $\ell$
and the non-trivial case is when $\ell=p$. Then the curve $E$ has good reduction at $m$, which can be
ordinary or supersingular. In either case, \rref{ech} is an isomorphism modulo $m$ by the theory of elliptic 
curves over finite fields. Thus, \rref{ech} is onto locally at $m$ by Nakayama's lemma. To show that \rref{ech}
is injective locally at $m$, suppose that $z$ is in its kernel. Then, by the Nullstellensatz, $x^n$ is divisible by 
$z$ for some $n\in \N$. In particular, for all $N\in \N$, 
\beg{efalse}{x^n=0\in R_{m}[[x]]/([p]x,x^N).}
The right hand side is calculable, and \rref{efalse} is false. Thus, \rref{ech} is an isomorphism locally, and
hence an isomorphism. 

It remains to show that $\mathcal{O}_{{}_pE}$ is finite flat over $R$. This is equivalent to showing that
for a maximal ideal $m$ of $R$, $(\mathcal{O}_{{}_pE})_m$ is a free $R_m$-module. Assuming again 
$\ell=p$, modulo $m$, it is, 
in fact, a free module on $p^2$ generators by the theory of elliptic curves over finite fields. Thus, 
$(\mathcal{O}_{{}_pE})_m$ is generated by $p^2$ elements as a module over $R_m$ by Nakayama's
lemma. If $z$ was a relation between these generators, then again $z$ becomes $0$ after inverting $x$,
so by the Nullstellensatz, $x^n$ is again divisible by $z$ for some $n$. But again, $x^n$ is not $0$ in 
$(\mathcal{O}_{{}_pE})_m$ for any $n\in \N$ (because it is not $0$ in the formal completion; note that it
{\em is} $0$ modulo $p$ for large enough $n$).
\end{proof}

\vspace{3mm}

\section{Construction of a $\Z/p$-equivariant ring spectrum}\label{sequivconst}

Our next aim is to realize the diagram \rref{epullback} in the category of (non-rigid) $\Z/p$-equivariant
commutative associative unital ring spectra. First, if $MU_*=L\r R$ is the homomorphism of rings
classifying the formal group law on $R$ (the elements $g_2$, $g_3$ are put in dimensions $4$, $6$, 
respectively), then by Landweber's exact functor theorem, 
$$\mathcal{E}_*X= MU_*X\otimes_{MU_*} R$$
defines a homology theory represented by a commutative associative ring spectrum $\mathcal{E}$
(see for example \cite{ahs}). 

\vspace{3mm}
Alternately, one can also use the following fact:

\begin{lemma}\label{lgci}
The classification map 
\beg{eclassif}{L\r R
}
of the formal group law on $R$ coming from the elliptic curve $E$, where $L$ is the Lazard ring, has
\beg{eclassif1}{x_1,x_2,x_3,x_5\mapsto 0,\; x_4\mapsto 8g_2,\; x_6\mapsto 48 g_3.
}
Thus, there exists a (rigid) $MU$-module spectrum $\mathcal{E}$ (in the sense of
\cite{ekmm}) realizing the module struture on 
coefficients given by \rref{eclassif}.
\end{lemma}

\begin{proof}
The relations \rref{eclassif1} are proved by direct computation. Therefore, a spectrum with coefficients
$\Z[1/6][g_2,g_3]$ can be obtained by killing, in the sense
of \cite{ekmm}, a regular sequence in the coefficients of the $E_\infty$ ring
spectrum $MU[1/6]$. Inverting $\Delta=g_2^3-27g_3^2$ gives the desired result.
\end{proof}

\vspace{3mm}
\noindent
{\bf Comment:} Note that the statement does not depend on the choices of the polynomial generators
of $MU_*$, since they are determined up to decomposables, and the indeterminacy, in the current case,
goes to $0$.

\vspace{3mm}

By Lemma \ref{lgci}, we then see that $\mathcal{E}\wedge \mathcal{E}$ splits as a coproduct of
suspensions of $\mathcal{E}$ as (rigid) $MU$-modules (the generators being induced, using 
functoriality of killing of regular sequences in modules in the sense of \cite{ekmm}, by maps of
rigid $MU$-modules
$$MU\r MU\wedge \mathcal{E}$$
given by homotopy classes in the target which are lifts of monomials in the Weierstrass strict isomorphism
parameters).

Using this, we can then realize the ring structure on $\mathcal{E}$ as a splitting of the coproduct 
in the category of rigid $MU$-modules which induces the ring structure on coefficients. Since all the maps
involved in the commutativity, associativity and unit diagrams then also lift to analogous maps for $MU$, 
one can then deduce the commutativity of those diagrams, in the derived category of rigid $MU$-modules,
from functoriality of killing a regular sequence in $MU$-modules.

\vspace{3mm}

Let us now consider the $\Z/p$-equivariant question.
The bottom row of diagram \rref{epullback} can be constructed
as Borel and Tate cohomology of $\mathcal{E}$ (with fixed $\Z/p$-action). To construct the right hand
column, denote by $\gamma$ the reduced regular (complex) representation of $\Z/p$. 
Then identify $S^{\infty\gamma}$ with its suspension spectrum. This is a $\Z/p$-equivariant $E_\infty$-ring
spectrum. The following is a well known fact (see for example \cite{lms}, Sections II 8,9).

\begin{lemma} \label{lgeom}
The fixed point functor and $?\wedge S^{\infty\gamma}$ induce inverse equivalences of categories
between the derived category of rigid $\Z/p$-equivariant $S^{\infty\gamma}$-modules and 
spectra. Moreover, with suitable models, these functors preserve the smash product rigidly.
\end{lemma}

\qed

\vspace{3mm}
Now after applying fixed points, the right hand column of the diagram \rref{epullback} is realized by a
ring map of Landweber-exact spectra due to the fact that $\mathcal{O}_{{}_pE}$ is a finitely generated 
projective $R$-module. Thus, we have a diagram of (non-rigid) commutative associative $\Z/p$-equivariant
ring spectra 
\beg{epull1}{\diagram
& S_1\dto\\
S_2\rto &S_3
\enddiagram
}
realizing, on $\Z/p$-coefficients the corresponding part of Diagram \ref{epullback}. (Note that on $\{e\}$-coefficients, 
the right hand column is $0$, and the bottom left term is $R$.) We want to argue that the homotopy limit
of \rref{epull1} is a commutative associative ring spectrum. To this end, note that the homotopy pullback is
weakly symmetric monoidal, so we have a canonical natural map from the smash product of two copies of 
the homotopy pullback of \rref{epull1}
to the homotopy pullback of
\beg{epull2}{\diagram
& S_1\wedge S_1\dto\\
S_2\wedge S_2\rto &S_3\wedge S_3.
\enddiagram
}
Now, there is certainly a map from the homotopy pullback of \rref{epull2} to the homotopy pullback of
\rref{epull1}, since a homotopy commutative map of pullback diagrams induces a map on homotopy pullbacks.
Furthermore, the indeterminacy of choosing such a map can be measured in the first (hence odd-degree)
 cohomology groups
$$S_3^1(S_1\wedge S_1),\;S_3^1(S_2\wedge S_2),$$
which are $0$.
This, together with a similar observation of the triple smash product analogue of \rref{epull2} implies that
the homotopy pullback of the diagram \rref{epull1} is a $\Z/p$-equivariant commutative associative ring 
spectrum $\mathcal{E}_{\Z/p}$. 
Note that the $\Z/p$-fixed points of $\mathcal{E}_{\Z/p}$ have coefficients $\mathcal{O}_{{}_pE}$
by Theorem \ref{t1}, and non-equivariant coefficients $R$, related by the canonical ring homomorphisms.

\vspace{3mm}

\section{Complex orientation. Equivariant formal groups.}\label{sco}

There exists a map, in the homotopy category of $\Z/p$-equivariant spectra, of the diagram
\beg{emudiag}{\diagram
& S^{\infty \gamma}\wedge MU_{\Z/p}\dto\\
F(E\Z/p_+,MU)\rto & S^{\infty\gamma}\wedge F(E\Z/p_+, MU)
\enddiagram}
into the diagram \rref{epull1} (which, on coefficients, realizes the diagram \rref{epullback}). Furthermore,
the diagram consists of maps of (non-rigid) ring spectra, and passes to a map of (non-rigid) ring spectra
after taking homotopy pullbacks. All these statements follow exactly by the same method as
the proof of the ring structure on $\mathcal{E}_{\Z/p}$. 

Perhaps a little more detail is due only in describing the map at the upper right corner. Recall \cite{zp} that
\beg{ephimu}{(S^{\infty\gamma}\wedge MU_{\Z/p})_*=MU_*[u_\alpha,u_\alpha^{-1},b_i^\alpha
\mid \alpha\in \Z/p^*\smallsetminus\{0\},\; i\in \N].
}
Denoting 
$$F(E\Z/p_+,MU)_*=MU_*[[x]]/[p]x,$$
under the right hand vertical arrow \rref{emudiag}, 
$$u_\alpha\mapsto [\alpha]x,$$
$$b_i^\alpha\mapsto \text{coeff}_{z^i}(z+_F[\alpha]x).$$
To realize the map from the upper right corner of \rref{emudiag} to the upper right corner of \rref{epull1},
it suffices to specify  where the classes $u_\alpha$, $b_i^\alpha$ go on coefficients. This is because 
(after using Lemma \ref{lgeom}),
\rref{ephimu} is realized by a wedge sum of copies of $MU$, the map on $MU$ is determined by classifying
the formal group law, and the map of (non-rigid) commutative ring spectra is uniquely determined by 
mapping the generators.

To lift the generators to $x^{-1}\mathcal{O}_{{}_pE}$, it suffices to consider the comultiplication
$$\mathcal{O}_U\r\mathcal{O}_U\otimes_R\mathcal{O}_{{}_pE},$$
and complete at the quotient morphism
$$\mathcal{O}_{{}_pE}\r \mathcal{O}_{{}_pE}\otimes_R \mathcal{O}_{{}_pE}.$$

\vspace{3mm}
We conclude that the spectrum $\mathcal{E}_{\Z/p}$
is complex-oriented, and thus possesses an equivariant formal group law (since the pullback of
Diagram \rref{emudiag} is $MU_{\Z/p}$, see \cite{zp}).

\begin{theorem}\label{tefgl}
We have
\beg{etefgl1}{\mathcal{E}_{\Z/p}^*(\C P^\infty_{\Z/p})=\mathcal{O}_{E^{\wedge}_{{}_pE}}\otimes_R 
\mathcal{O}_{{}_pE}
}
where the right hand side denotes the coefficient ring of the formal completion of $E$ at ${}_pE$ (an affine
formal scheme). Additionally, the following diagram of rings is strictly Cartesian:
\beg{etefgl2}{\diagram
\mathcal{O}_{E^{\wedge}_{{}_pE}}\otimes_R \mathcal{O}_{{}_pE}\rto\dto &\displaystyle\prod_{\alpha\in \Z/p^*}(x^{-1}\mathcal{O}_{{}_pE}
)[[z_\alpha]]
\dto \\
R[[x,z]]/[p]x\rto & \displaystyle \prod_{\alpha\in \Z/p^*} (x^{-1}R[[x]]/[p]x)[[z_\alpha]].
\enddiagram
}
Here $\Z/p^*$ denotes
the group of characters of $\Z/p$.
The left vertical arrow of \rref{etefgl2}
sends the coordinate on $E$ given by the element $x$ of \rref{er0} to $z$ (a conflict
of notation is caused here by the fact that we already denoted by $x$ the same coordinate on $\mathcal{O}_{{}_pE}$).
We put
\beg{ezaza}{z_\alpha=z+_F [\alpha]x.}
The horizontal maps are products, over $\alpha$, of the maps given by \rref{ezaza}.
\end{theorem}

\begin{proof}
The equivariant infinite complex projective space is the space of complex lines in a complete $G=\Z/p$-universe
$\mathcal{U}$. The universe can be filtered by finite-dimensional representations $V$, and thus, it follows
from the complex orientation that
$$\mathcal{E}^*_{\Z/p}(P(V))$$
is a free $\mathcal{E}^*_{\Z/p}$-module on monomials 
$$1, z_{\alpha_1},z_{\alpha_1}z_{\alpha_2},\dots, z_{\alpha_1}z_{\alpha_2}\dots z_{\alpha_{k-1}}$$
where $\alpha_1$, $\alpha_1\oplus \alpha_2$,\dots $\alpha_1\oplus\alpha_2\oplus\dots\oplus
\alpha_k$ is some complete
equivariant flag in $V$. 

On the other hand, the ring $\mathcal{O}_{{}_pE}$ is locally complete intersection. To see this, 
it suffices to proceed locally, and in fact to complete (\cite{matsu}, Theorem 21.2). In particular, we can
replace $R$ by its completion $\widehat{R}$ at a maximal ideal. Then, for example by \cite{tate},
locally, $\mathcal{O}_{{}_pE}$ is isomorphic to a formal group, i.e. a ring of the form 
$\widehat{R}[[x]]/[p]x$, so it is a locally complete intersection.

Now since $\mathcal{O}_{{}_pE}$ is a locally complete intersection, by \cite{matsu}, Theorem 21.2,
locally, it is a quotient of, say, $\mathcal{O}_U$, by a principal ideal $(q)$, since $\mathcal{O}_U$ is an 
integral domain, $(q^n)/(q^{n+1})\cong \mathcal{O}_U/(q)$ as $\mathcal{O}_U$-modules.

Thus, \rref{etefgl1} follows by taking limits.

Next, we prove that the right vertical arrow and the bottom horizontal arrow of \rref{etefgl2} 
are jointly onto. Consider an element 
$$u_0=(u_{0,\alpha})_\alpha\in \prod_{\alpha\in \Z/p^*} (x^{-1}R[[x]]/[p]x)[[z_\alpha]].$$
Then by the Chinese remainder theorem, there exists an element 
$$v_0\in (x^{-1}R[[x]]/[p]x)[z]$$
which is equal to $u_0$ modulo $\displaystyle \prod_\alpha z_\alpha$. Additionally,
since the diagram \rref{epullback} is strictly cartesian, the element $v_0$ can be written as a sum
of an element 
$$w_0\in (R[[x]]/[p]x)[z]$$
and an element 
$$q_0\in \prod_\alpha (x^{-1}\mathcal{O}_{{}_pE})[z_\alpha].$$
Now choose
$$u_1\in \prod_{\alpha\in \Z/p^*} (x^{-1}R[[x]]/[p]x)[[z_\alpha]]$$
so that
$$u_1\prod_\alpha z_\alpha = u_0-v_0-w_0$$
and repeat this procedure.

Now the diagram \rref{epullback} is the coefficients of the right square 
of the ``Tate diagram" of the $\Z/p$-equivariant spectrum $\mathcal{E}_{\Z/p}$, i.e. the diagram
$$\resizebox{0.9\hsize}{!}{\diagram
E\Z/p_+\wedge \mathcal{E}_{\Z/p}\dto_\sim\rto &\mathcal{E}_{\Z/p}\dto \rto & S^{\infty\gamma}\wedge
\mathcal{E}_{\Z/p}\dto \\
E\Z/p_+\wedge F(E\Z/p_+,\mathcal{E}_{\Z/p})\rto &
F(E\Z/p_+,\mathcal{E}_{\Z/p})\rto &
S^{\infty\gamma}\wedge F(E\Z/p_+,\mathcal{E}_{\Z/p}).
\enddiagram}
$$
One obtains the diagram \rref{etefgl2} by applying $F((\C P^\infty_{\Z/p})_+,?)$.
The rows are cofibration sequences. It was proved in \cite{zp} that when the right column
and the bottom row of the right square of such a diagram are onto on coefficients, the diagram is
strictly Cartesion on coefficients.

\end{proof}


\vspace{5mm}

\section{Equivariant formal groups from $p$-divisible groups}\label{spdiv}

It is pretty clear from the onstructions we performe so far that the passage from an elliptic curve to
a $\Z/p$-equivariant spectrum only depends on its $p$-divisible group, or, more precisely,
its pullback by the simple $p$-torsion subgroup of the \'{e}tale part.
(In fact, as we will see later, the same is true for $\Z/p^n$ if we pull back to $p^n$-torsion of 
the \'{e}tale part.) Because of this,
let us adapt the discussion to that context.

Let $R$ be a complete Noetherian local ring of characteristic $0$ whose residue
field is a perfect field of characteristic $p$. Let $G$ be a $p$-divisible group (also known
as Barsotti-Tate group) over $Spec(R)$. Then by \cite{tate},
we have a short exact sequence
\beg{espdiv}{0\r G_0\r G\r G_{et}\r 0}
where $G_0$ is connected and corresponds to a formal group (which we
will identify with $G_0$), and $G_{et}$ is \'{e}tale. We will work with
the following 
\beg{eass}{\parbox{3.5in}{{\bf Assumption:} The formal group $G_0$ has a $1$-dimensional summand
$G^\prime_0$.}
}
Then we have a diagram
$$\diagram
0\rto & G_0\rto\dto & G\rto\dto & G_{et}\dto^{=}\rto &0\\
0\rto & G^\prime_0\rto & G^\prime\rto & G_{et}\rto & 0
\enddiagram
$$
where the left square is a pushout. 
Denoting by ${}_pG$ the $p$-torsion subgroup of $G$, we obtain a short exact sequence
$$0\r {}_{p^r}G^\prime_0\r {}_{p^r}G^\prime\r {}_{p^r}G^\prime_{et}\r 0.$$
Now we have a diagram where the right square is a pullback:
$$\diagram
0\rto & G^\prime_0\dto _{=}\rto &G^{\prime\prime}\dto\rto & {}_{p^r}G^{\prime}_{et}\dto\rto & 0\\
0\rto & G^\prime_0\rto & G^\prime\rto & G_{et}\rto &0.
\enddiagram
$$
Now consider the diagram
\beg{ebigdig}{\diagram
&0\dto&0\dto&&\\
0\rto &{}_{p^r}G^\prime_0\dto\rto & {}_{p^r}G^\prime\dto\rto &{}_{p^r}G^\prime_{et}\dto^=\rto & 0\\
0\rto & G^\prime_0\rto\dto & G^{\prime\prime}\dto\rto & {}_{p^r}G^\prime_{et}\rto & 0\\
&G^\prime_0/{}_{p^r}G^\prime_0\rto_=\dto &G^\prime_0/{}_{p^r}G^\prime_0.\dto &&\\
&0&0&&
\enddiagram
}

\begin{construction}\label{tpdiv}
Fix a parameter $z\in \mathcal{O}_{G^{\prime\prime}}$ such that 
$$\mathcal{O}_{G^{\prime\prime}}/(z)=R$$
(note that such a parameter always exists, since $\mathcal{O}_{G_0^\prime}$ is a power series
ring in one generator which can be extended to be non-zero over the non-zero part of ${}_pG^\prime_{et}$).
This data specifies a $\Z/p^r$-equivariant formal group law on
\beg{ebigdig1}{\mathcal{O}_{G^{\prime\prime}}\widehat{\otimes}_{R}
\mathcal{O}_{{}_{p^r}G^\prime}}
over $\mathcal{O}_{{}_{p^r}G^\prime}$.
\end{construction}

\begin{proof}
If we denote
$$R[[z]]=\mathcal{O}_{G_0^\prime},$$
and put $t=[p^r]z$, Then we can write
$$R[[t]]=\mathcal{O}_{G^\prime_0/{}_{p^r}G^\prime_0},$$
which, by Diagram \rref{ebigdig}, after pulling back over ${}_{p^r}G^\prime_{et}$ to
${}_{p^r}G^\prime$
specifies a ``good parameter" (in the sense of Strickland \cite{stricmult})
on $G^{\prime\prime}\times_{{}_pG^\prime_{et}}{}_{p^r}G^\prime$. The required map
$${}_{p^r}G^\prime\times \Z/p^*\r G^{\prime\prime}\times_{Spec(R)}{}_{p^r}G^\prime$$
over ${}_{p^r}G^\prime$ is given by sending 
$${}_{p^r}G^\prime\times \{\alpha\}\r G^{\prime\prime}\times_{Spec(R)}{}_{p^r}G^\prime$$
via $\alpha$ times the embedding. The required properties are easily verified.
\end{proof}

\vspace{3mm}
\noindent
{\bf Comment:} This type of $\Z/p^r$-equivariant formal group law is not general, not only because
we assumed that $R$ is a complete local ring. It has, in fact, extra structure. In the notation of Strickland 
\cite{stricmult}, Definition 2.16, the scheme $S$ is in fact a commutative group scheme over another
scheme $T$ (in our case, $T=Spec(R)$), and in the map 
$$\phi:(\Z/p^r)^*\times S\r C,$$
for $\alpha\in (\Z/p^r)^*$, 
$$\phi(\alpha,?):S\r C,$$
where $C$ is the formal group scheme involved in the definition \cite{stricmult} of an equivariant FGL,
is in fact
a homomorphism of formal group schemes over $T$. This type of equivariant formal group law 
can
be reinterpreted in terms of level structures on moduli spaces of group schemes. From that point
of view, the
present example of $p$-divisible groups, was, in fact, treated explicitly by Peter Scholze \cite{scholze}.

\vspace{3mm}
\section{$\Z/p$-equivariant Barsotti-Tate cohomology}\label{sbt}

Let $k$ be a perfect field of characteristic $p$ and let $G$ be a $p$-divisible group. 
Then the short exact sequence \rref{espdiv} splits canonically, so we may as well just take
as our data the formal group 
$G_0$  and the \'{e}tale group $G_{et}$. Suppose $G_0$ has height $h$ and dimension $1$
and $G_{et}$ has height $k$. The Lubin-Tate's theorem \cite{lt} has the following extension

\begin{theorem}\label{tbl} (\cite{bl}, Theorem 7.1.3.)
Deformations $\widetilde{G}$ of $G$ to complete local rings with residue field $k$ are represented by the ring
\beg{ewdef}{W(k)[[u_1,\dots,u_{h-1}]][[w_1,\dots,w_k]]}
where $W(k)$ denotes the ring of Witt vectors of $k$.
\end{theorem}

\qed

\vspace{3mm}
In \rref{ewdef}, 
$$W(k)[[u_1,\dots,u_{h-1}]]$$
represents the deformations $\widetilde{G}_0$
of $G_0$. The deformation $\widetilde{G}_{et}$ of $G_{et}$ is unique, and
the parameters $w_1,\dots,w_k$ classify the extension. Non-trivial extensions arise,
roughly, by the following principle: Suppose $p$ times a closed point $Q$ of $\widetilde{G}_{et}$ is $0$ in 
$\widetilde{G}_{et}$. Then the extension specifies a closed point
$[p]_{\widetilde{G}}Q$ in $\widetilde{G}_0$. Now $\widetilde{G}_0$ is also $p$-divisible, so $[p]Q=[p]P$ for some
point $P$ of $\widetilde{G}_0$. Thus, we have $[p](Q-P)=0$, but the point $Q-P$ is not defined
in an \'{e}tale extension, since $P$ is, in fact, totally ramified. 

This discussion must be carried out on coordinates (instead of points) to give a rigorous argument, but it
shows that even ${}_p\widetilde{G}$ is a non-trivial extension of its formal part by its \'{e}tale part.

\vspace{3mm}
Now Goerss and Hopkins \cite{gh} proved that there exists an $E_\infty$-ring spectrum $E_h$ with
coefficients
\beg{ehcoeff}{E_{h*}=W(k)[[u_1,\dots,u_{h-1}]][u,u^{-1}]
}
which is $H_\infty$-complex-oriented by Ando \cite{ando}, and its complex orientation, on coefficients, is the map from the Lazard ring
to \rref{ehcoeff} which classifies the universal deformation of the height $h$ Lubin-Tate formal group law
\cite{lt}. Since the additional generators are in degree $0$, 
by the same method, one can construct a non-equivariant $E_\infty$-ring spectrum $E_{(h,k)}$ with
\beg{ehkcoeff}{E_{(h,k)*}=W(k)[[u_1,\dots,u_{h-1},w_1,\dots,w_k]][u,u^{-1}]
}
together with an $H_\infty$-complex orientation which, on coefficients, factors through the inclusion
from \rref{ehcoeff} to \rref{ehkcoeff}. Now let $(\mathcal{O}_{{}_p\widetilde{G}})_{(0)}$,
$(\mathcal{O}_{{}_p\widetilde{G}_0})_{(0)}$ denote the localizations of the rings
$\mathcal{O}_{{}_p\widetilde{G}}$, $\mathcal{O}_{{}_p\widetilde{G}_0}$ away from the zero section 
of $\widetilde{G}$. Then we have (in this case, automatically) a pullback diagram of the form
\beg{ebtpull}{
\diagram
\mathcal{O}_{{}_p\widetilde{G}}\rto\dto & (\mathcal{O}_{{}_p\widetilde{G}})_{(0)}\dto\\
\mathcal{O}_{{}_p\widetilde{G}_0}\rto & (\mathcal{O}_{{}_p\widetilde{G}_0})_{(0)}.
\enddiagram
}
Our aim is to realize the diagram \rref{ebtpull}, (after attaching $[u,u^{-1}]$) by coefficients of equivariant
$E_\infty$-ring spectra. 

In effect, for the terms of the bottom row of \rref{ebtpull}, we can just take
$F(E\Z/p_+,E_{(h,k)})$ and $S^{\infty \gamma}\wedge F(E\Z/p_+,E_{(h,k)})$. On the other hand, 
for the right column, we can use Lemma \ref{lgeom}. The ring $\mathcal{O}_{{}_p\widetilde{G}}$
is a finitely generated free module over \rref{ewdef}. Thus, we can construct a non-equivariant spectrum
with coefficients $\mathcal{O}_{{}_p\widetilde{G}}[u,u^{-1}]$, and give it $E_\infty$-ring structure
using \cite{gh}. Localizing away from the pullback of the $0$-section of $\widetilde{G}$
can be done in the $E_\infty$-category, as well as constructing an $E_\infty$ map to the fixed point
spectrum of $S^{\infty \gamma}\wedge F(E\Z/p_+,E_{(h,k)})$. Using Lemma \ref{lgeom}, then,
we obtain an $S^{\infty\gamma}$-$E_\infty$-algebra $E_{(h,k,0)}$ and a $\Z/p$-equivariant
$E_\infty$-map 
$$E_{(h,k,0)}\r S^{\infty \gamma}\wedge F(E\Z/p_+,E_{(h,k)})$$
realizing, in the category of $E_\infty$-$\Z/p$-ring spectra, 
the right column of \rref{ebtpull} (after adjoining $[u,u^{-1}]$). We call the homotopy pullback of
the $\Z/p$-equivariant $E_\infty$-realization of the lower row and right column of \rref{ebtpull} the 
{\em $\Z/p$-equivariant Barsotti-Tate cohomology $\mathcal{R}$}. It realizes the $\Z/p$-equivariant
formal group law of Theorem \ref{tpdiv} in the case when $R$ is the ring 
$\mathcal{O}_{{}_p\widetilde{G}}[u,u^{-1}]$, and $G^\prime$ is
the $p$-divisible group $\widetilde{G}$.

\vspace{3mm}

\section{The case of $\Z/p^{r}$ for $r>1$.}\label{szpr}

There is a more involved but analogous picture of all of the above discussion for
$\Z/(p^r)$-equivariant spectra. First, let us recall the structure of $\Z/p^r$-equivariant
complex cobordism. For a finite abelian group $G$ and a subgroup $H\subseteq G$, we have the
family $\mathcal{F}[H]$ of subgroups of $G$ not containing $H$. For $G=\Z/p^r$, we have
a chain of families
$$\mathcal{F}[\{e\}]\subset
\mathcal{F}[\Z/p]\subset\mathcal{F}[\Z/p^2]\subset\dots\subset \mathcal{F}[\Z/p^{r+1}].$$
(The first family is empty and the last one is, by convention, the family of all subgroups of $\Z/p^r$.)
It is also convenient to write 
$$E\Z/p^{r-i}=E\mathcal{F}[\Z/p^{i+1}].$$

Now it is proved in \cite{hu} that any $\Z/p^r$-equivariant spectrum $X$ is equivalent to
the homotopy limit of the finite diagram
\beg{ezpreq}{\resizebox{0.9\hsize}{!}{\diagram
&&\vdots\dto\\
&\widetilde{E\mathcal{F}[\Z/p^i]}\wedge F(E\mathcal{F}[\Z/p^{i+1}]_+,X)\dto\rto &
\widetilde{E\mathcal{F}[\Z/p^{i+1}]}\wedge F(E\mathcal{F}[\Z/p^{i+1}]_+,X)\\
\dots \rto & \widetilde{E\mathcal{F}[\Z/p^i]}\wedge F(E\mathcal{F}[\Z/p^{i}]_+,X). &
\enddiagram}
}
(This is a version of {\em isotropy separation}.) The first term on the lower left is 
$$\widetilde{E\mathcal{F}[\Z/p^0]}\wedge F(E\mathcal{F}[\Z/p]_+,X)=F(E\Z/p^r_+,X),$$
whose fixed points are the Borel cohomology of $X$, the last term on the upper right is
$$\widetilde{\mathcal{F}[\Z/p^r]}\wedge F(E\mathcal{F}[\Z/p^{r+1}]_+,X)=
\widetilde{\mathcal{F}[\Z/p^r]}\wedge X,$$
whose fixed points are the ``geometric fixed points" $\Phi^{\Z/p^r}X$. For any $\Z/p^r$-equivariant
spectrum $X$, the homotopy limits of  ``partial staircase diagrams" obtained by omitting an upper right
(resp. a lower left) section of \rref{ezpreq} are of the form
$$F(E\mathcal{F}[\Z/p^i]_+, X),$$
resp.
$$\widetilde{E\mathcal{F}[\Z/p^i]}\wedge X$$
with varying $i$ (see \cite{hu}).

\vspace{3mm}
Furthermore, in the case of $X=MU_{\Z/p^r}$, all of the terms of \rref{ezpreq}
are computable, and the limit of the diagram \rref{ezpreq} on coefficients is equal to the 
coefficients of $MU_{\Z/p^r}$ (without higher derived functors), \cite{hu}. The
term 
$(\widetilde{E\mathcal{F}[\Z/p^i]}\wedge F(E\mathcal{F}[\Z/p^{i+1}]_+,MU_{\Z/p^r}))_*$
is calculated explicitly in \cite{hu}. The main idea is that inductively, we can calculate 
$(MU_{\Z/p^i})_*$ and by the Borel cohomology spectral sequence, an associated graded
object of $F(E\mathcal{F}[\Z/p^{i+1}]_+,MU_{\Z/p^r})_*$
is
$$(MU_{\Z/p^i})_*[[u_{p^i}]]/[p^{r-i}]u_{p^i}.
$$
(One can be explicit about the extension \cite{hu}.) The desired term
$(\widetilde{E\mathcal{F}[\Z/p^i]}\wedge F(E\mathcal{F}[\Z/p^{i+1}]_+,MU_{\Z/p^r}))_*$
is obtained by inverting the Euler class $u_{p^{i-1}}\in (MU_{\Z/p^i})_*$.
One then has
$$
\begin{array}{l}
(\widetilde{E\mathcal{F}[\Z/p^{i+1}]}\wedge F(E\mathcal{F}[\Z/p^{i+1}]_+,MU_{\Z/p^r}))_*=\\
u_{p^i}^{-1}(\widetilde{E\mathcal{F}[\Z/p^i]}\wedge F(E\mathcal{F}[\Z/p^{i+1}]_+,MU_{\Z/p^r}))_*.
\end{array}
$$
A subtle point is that the ``staircase diagram" which occurs in \cite{ak} is actually
different. To prevent confusion, we clarify the relationship in Appendix I (Section \ref{sapp}).

\vspace{3mm}
Now we can construct commutative complex-oriented $\Z/p^r$-equivariant ring 
spectra with coefficient rings $\mathcal{O}_{{}_{p^r}E}$ where $E$ is as in Section \ref{sell} 
(elliptic cohomology), or
$\mathcal{O}_{{}_{p^r}\widetilde{G}}[u,u^{-1}]$ (Barsotti-Tate cohomology)
where $\widetilde{G}$ is as in Section \ref{sbt}
as homotopy limits of diagrams of the form \rref{ezpreq}, by constructing the appropriate diagrams of the
form 
\beg{ggeomr}{\diagram
&&\vdots\dto\\
& R_i\dto\rto & R^\prime_i\\
\dots \rto & R^\prime_{i-1} &
\enddiagram
}
on rings of
coefficients. Additionally, in the case of Barsotti-Tate cohomology, we can realize the diagram in question
(and hence its homotopy limit) in the category of $\Z/p^r$-equivariant $E_\infty$-ring spectra $\mathcal{R}_i$.

The analogue of Lemma \ref{lgeom} in the case of $\Z/p^r$ is the following result. 
Recall first that $\Z/p^r$-equivariantly, $\widetilde{E\mathcal{F}[\Z/p^i]}$ 
has a model as an $E_\infty$-ring 
space, namely $S^{\infty V}$ where $V$ is the sum of all non-isomorphic irreducible $\Z/p^r$-representations
which have trivial $\Z/p^i$-fixed points.

\begin{lemma}\label{geomr}
The $\Z/p^i$-fixed point functor and $?\wedge S^{\infty V}$ induce inverse equivalences of categories between
the derived category of $\Z/p^r$-equivariant $S^{\infty V}$-module spectra and $\Z/p^{r-i}$-equivariant 
spectra.
\end{lemma}

\qed

Note that this result is particularly well-suited to realizing staircase-shaped diagrams of $\Z/p^r$-equivariant
spectra of the form \rref{ezpreq}, which arise from applying isotropy separation
to any $\Z/p^r$-equivariant spectrum $X$. This is because the $\Z/p^r$-equivariant spectrum 
\beg{ecomponent}{\widetilde{E\mathcal{F}[\Z/p^i]}\wedge F(E\mathcal{F}[\Z/p^{i+1}]_+,X)}
by Lemma \ref{geomr}, by taking $\Z/p^i$-fixed points, corresponds to a $\Z/p^{r-i}$-equivariant spectrum,
which encodes all of its information. For
$i>0$, we have $r-i<r$, while for $i=0$, \rref{ecomponent} is just Borel cohomology. This allows us to
proceed by induction.  A subtle point is that in general, the $\Z/p^{r-i}$-equivariant
spectrum corresponding to \rref{ecomponent} by Lemma \ref{geomr} is {\em not}
just the Borel cohomology $F((E\Z/p^{r-i})_+,X)$. This is because in Lemma \ref{geomr}, the functor
from $\Z/p^{r-i}$-equivariant spectra to $S^{\infty V}$-modules really means smashing $S^{\infty V}$ 
with the {\em inflation} of a $\Z/p^{r-i}$-equivariant spectrum to $\Z/p^r$. In \rref{ecomponent},
on the other hand, $F(E\mathcal{F}[\Z/p^{i+1}]_+,X)$ is understood as a $\Z/p^r$-equivariant spectrum,
which is not an inflation. By a ``miracle" (connectivity of generators), in the case of cobordism, the 
isotropy-separated spectra \rref{ecomponent} are already complete for $i=1$ (see Appendix I and
\cite{hu}). We will see that in other examples
of interest, this is not the case, and an algebraic localization occurs in writing down the coefficients
of \rref{ecomponent}. 

\vspace{3mm}

\vspace{3mm}

In the case of Barsotti-Tate cohomology, considering the $p$-divisible group $\widetilde{G}$ as in
Section \ref{sbt}, 
consider the short exact sequence
$$0\r ({}_{p^r}\widetilde{G})_0\r {}_{p^r}\widetilde{G}\r ({}_{p^r}\widetilde{G})_{et}\r 0.$$
Let $({}_{p^r}\widetilde{G})_{(i)}$ be the pullback 
$$
\diagram
({}_{p^r}\widetilde{G})_{(i)}\rto \dto & ({}_{p^i}\widetilde{G})_{et}\dto\\
{}_{p^r}\widetilde{G}\rto & ({}_{p^r}\widetilde{G})_{et}.
\enddiagram
$$
We let the ring $R_i$, $i=0,\dots, r$ be the coefficient ring of $({}_{p^r}\widetilde{G})_{(i)}$ localized away
from ${}_{p^{i-1}}\widetilde{G}$ (in the case $i=0$, 
there is no localization). No completion is needed here, since the rings in question are already complete.
The ring $R_{i}^\prime$ is the coefficient ring of  
$({}_{p^r}\widetilde{G})_{(i)}$, localized away
from ${}_{p^{i}}\widetilde{G}$, $i=0,\dots, r-1$.

Algebraically, ${}_{p^r}\widetilde{G}_{(i)}$ and ${}_{p^r}\widetilde{G}_{(i+1)}\smallsetminus
{}_{p^i}\widetilde{G}$ are open sets in the affine scheme ${}_{p^r}\widetilde{G}_{(i+1)}$ whose
intersection is ${}_{p^r}\widetilde{G}_{(i)}\smallsetminus {}_{p^i}\widetilde{G}$:
$$
\diagram
{}_{p^r}\widetilde{G}_{(i)}\rto^\subset & {}_{p^r}\widetilde{G}_{(i+1)}\\
{}_{p^r}\widetilde{G}_{(i)}\smallsetminus {}_{p^i}\widetilde{G}\uto^\subset
\rto_\subset  &
{}_{p^r}\widetilde{G}_{(i+1)}\smallsetminus
{}_{p^i}\widetilde{G}\uto^\subset
\enddiagram
$$
Thus, we have a staircase-shaped diagram of affine schemes
\beg{eelliptic1000}{
\resizebox{0.9\hsize}{!}{\diagram
{}_{p^r}\widetilde{G} &\dots & {}_{p^r}\widetilde{G}\smallsetminus {}_{p^i}\widetilde{G} &
{}_{p^r}\widetilde{G}\smallsetminus {}_{p^{r-1}}\widetilde{G}\\
\vdots &&{}_{p^r}\widetilde{G}_{(i+1)}\smallsetminus {}_{p^i}\widetilde{G}&\dots\lto\uto\\
{}_{p^r}\widetilde{G}_{(i)} &
{}_{p^r}\widetilde{G}_{(i)}\smallsetminus {}_{p^{i-1}}\widetilde{G} &
{}_{p^r}\widetilde{G}_{(i)}\smallsetminus {}_{p^{i}}\widetilde{G}\lto\uto &\\
{}_{p^r}\widetilde{G}_{(0)} &\dots\uto\lto &&
\enddiagram
}
}
where the unconnected terms are colimits of the parts of the diagram restricted to the area between
the given row and column.
In other words, the colimits of the partial staircase diagrams of these schemes obtained by omitting an upper right
section are ${}_{p^r}\widetilde{G}_{(i)}$. Similarly, one sees that if we omit 
a lower left section of the staircase diagram, the colimit of the resulting  truncated staircase diagram is ${}_{p^r}\widetilde{G}\smallsetminus {}_{p^i}\widetilde{G}$. Note that all these schemes are affine.
Our task is to find a $\Z/p^r$-equivariant spectrum $X$ such that the coefficients of the fixed points
in the diagram \rref{ezpreq} are the coefficient rings of the affine schemes \rref{eelliptic1000}.

Topologically, our approach to constructing the equivariant spectrum $X$ is as follows: Note that
a $\Z/p^r$-equivariant spectrum $Y$, has a canonical action of $\Z/p^r$ by morphisms of $\Z/p^r$-equivariant 
spectra. 
We shall construct inductively $\Z/p^r$-equivariant spectra $\mathcal{R}=\mathcal{R}_{\Z/p^r}$
whose naive $\Z/p^r$ group action extends to a $\Z/p^\infty$-action by $\Z/p^r$-equivariant morphisms. For $r=0$,
this naive action is trivial. 

Thus, the
constituent spectrum \rref{ecomponent} can be constructed by taking $(\mathcal{R}_{\Z/p^i})^{\Z/p^i}$
with its naive $\Gamma=(\Z/p^\infty)/(\Z/p^i)$-action, take Borel cohomology with respect to $\Z/p^{r-i}
\subset\Gamma$, localize away from ${}_{p^r}\widetilde{G}_{(i-1)}$, and then smash with 
$\widetilde{E\mathcal{F}[\Z/p^{i+1}]}$ (see Lemma \ref{geomr}). All of the structure is recovered, the
``staircase-shaped" maps \rref{ggeomr} are realized by functoriality, and
the construction can be carried out $E_\infty$.

\vspace{3mm}
In the case of elliptic cohomology (for $p>3$), we can proceed similarly, with the caveat that the rings 
$\mathcal{O}_{{}_{p^r}E}$
are not complete, and hence for $R_i$, we must take
$$(\mathcal{O}_{{}_{p^r}E})^\wedge_{{}_{p^i}E},$$
while $R^\prime_i$ is the localization of $R_i$ away from ${}_{p^{i-1}}E$. 

On the algebraic
side, we observe similarly as in Section \ref{sell} that 
$${}_{p^r}E\subset U$$
(for the same reason). Thus, we can construct the rings $\mathcal{O}_{{}_{p^r}E}$ as
quotients of the ring $R_0$ of \rref{er0} by suitable ideals. Next, we observe that those ideals are
principal by the fact that $Pic_0(E)=E$, and the fact that the sum of all $p^r$-torsion points of $E$ is $0$.
Then the proof that the diagram
\beg{ediagpr}{\diagram
({}_{p^r}E)^\wedge_{{}_{p^i}E} & ({}_{p^r}E)^\wedge_{{}_{p^i}E}\smallsetminus {}_{p^{i-1}}E\lto\\
({}_{p^r}E)^\wedge_{{}_{p^{i-1}}E}\uto &
({}_{p^r}E)^\wedge_{{}_{p^i}E}\smallsetminus {}_{p^{i-1}}E\uto\lto
\enddiagram
}
is strictly Cartesian on coefficients proceeds the same way as the proof of Theorem \ref{t1}. (Using
the results of Greenlees 
\cite{greenleestate}, it can be alternately deduced simply from the fact that there is no
infinite torsion with respect to the generator of the ideal of ${}_{p^i}E$, which 
follows from the fact that the rings involved are Noetherian.) 

On the topological side, we proceed similarly as in the Barsotti-Tate case, constructing $\Z/p^r$-equivariant
spectra $\mathcal{E}_{\Z/p^r}$
whose naive $\Z/p^r$-equivariant structure extends to a naive $\Z/p^\infty$-equivariant structure.
Again, for $r=0$, both actions are trivial. For the constituent $\Z/p^r$-equivariant spectra \rref{ecomponent},
again, we take the $\Z/p^{r-i}$-Borel cohomology of $(\mathcal{E}_{\Z/p^i})^{\Z/p^i}$, localized away
from the generator $w$ of the ideal of ${}_{p^{i-1}}E$. (Again, here we consider 
$\Z/p^{r-i}\subset (\Z/p^\infty)/(\Z/p^i)$. The computation of Borel cohomology involves the fact
that the coefficient ring of ${}_{p^r}E$ has no higher $v$-torsion where $v$ is the generator of
the ideal corresponding to ${}_{p^i}E$. This is due to the fact that the torsion subgroups of elliptic curves
are reduced schemes. Note that the fact that we are taking Borel cohomology here is what forces us to take
the completion in \rref{ediagpr}.

\vspace{3mm}
Thus, we have

\vspace{3mm}
\begin{theorem}\label{tbt10}
There exists a $\Z/p^r$-equivariant $E_\infty$-ring spectrum $\mathcal{R}$ with coefficients
$\mathcal{O}_{{}_{p^r}\widetilde{G}}[u,u^{-1}]$ where $\widetilde{G}$ is the $p$-divisible group
of Theorem \ref{tbl} over the ring \rref{ehkcoeff}, which is complex-oriented, realizing the $\Z/p^r$-equivariant
formal group law of Theorem \ref{tpdiv}.

For the elliptic curve $E$ mentioned in Section \ref{sell}, there exists a complex oriented
$\Z/p^r$-equivariant commutative
ring spectrum $\mathcal{E}$ 
with coefficients $\mathcal{O}_{{}_{p^r}E}[u,u^{-1}]$, realizing the $\Z/p^r$-equivariant formal group law 
associated with $E$ by the construction of Strickland \cite{stricmult}.
\end{theorem}

\qed

\vspace{3mm}
Note that all of our constructions of the 
$\Z/p^r$-equivariant formal group law, and $\Z/p^r$-equivariant ring spectrum realizing this equivariant
formal group law, are determined by the ring $\mathcal{O}_{{}_{p^r}\widetilde{G}}$, and the ``Euler class" 
$x\in \mathcal{O}_{{}_{p^r}\widetilde{G}}$ corresponding to the generator of the character group $(\Z/p^r)^*$.
This is not surprising, of course, since for a finite $p$-group $\Gamma$ over $R$ which has a short exact 
sequence
$$0\r \Gamma_{(0)}\r \Gamma\r \Gamma_{et}\r 0$$
where $\Gamma_{(0)}$ is the connected part and $\Gamma_{et}$ is the \'{e}tale part is split, $\Gamma$ is split if and only 
if the $0$-section has a point over the non-zero components of $\Gamma_{et}$.

\vspace{3mm}
In the case of an elliptic curve over a complete Noetherian local ring in mixed characteristic with perfect residue field, 
the equivariant formal group laws and spectra constructed here are 
trivial: if the curve is supersingular, the $p$-divisible group of $p$ power torsion points is formal. If the curve
is ordinary, then its connected and \'{e}tale parts both have height $1$. Thus, the \'{e}tale part splits over
an unramified extension, while the formal part splits over a cyclotomic totally ramified extension. Thus, the splitting
Galois group is abelian, and the extension of the \'{e}tale part by the formal part is split. 

\vspace{3mm}
On the other hand, by \cite{deu}, for every prime $p>2$, there exists an elliptic curve $E$ whose group of 
points over $\mathbb{F}_p$ is $\Z/p$ (and thus is ordinary). Since for $p>7$, no rational curve
has $p$-torsion by Mazur's theorem, the group of $p$-torsion points of $E$ does not split over $\Z_{(p)}$ (
note that since
an elliptic curve is a projective variety, points over $\Q$ are the same thing as points over $\Z$).

\vspace{3mm}
Back to $p$-divisible groups, it is worth noting that we obtain no non-trivial equivariant spectra
from the ring of integers of $p$-adic complex numbers
$\mathcal{O}_{\mathbb{C}_p}$. By the classification of Scholze-Weinstein \cite{sw}, Theorem B, in this
case, a $p$-divisible group $G$ of dimension $d$ and height $h$
is classified by a rank $h$ free $\Z_p$-module $T\cong \Z_p^h$, and
a $d$-dimensional $\C_p$-vector subspace of $T\otimes_{\Z_p}\C_p$. The $T$ corresponds to the
Tate module of $G$, $W$ to its Lie algebra. The map is obtained from the fact that by duality, an
element of the Tate module of the Cartier dual 
$G^\prime$ determines a map $G\r \mathbb{G}_m(p)$, whose $p$-adic logarithm gives a map $Lie(G)
\r \C_p$.

Now by functoriality, $G_0$ is classified by the pair $(W,T_0)$ where $T_0\subseteq T$ is the $\Z_p$-submodule
$$\bigcap \{U\subseteq T\otimes_{\Z_p}\Q_p\mid W\subseteq U\otimes_{\Q_p} \C_p\}\cap T.$$
Now $G_0$ is a formal group (its Rapoport-Zink parameters are obtained by a reduced row echelon form
basis of $W$ whose entries have non-negative valuation), and since $T_0$ is a direct summand of $T$, $G_0$ is 
a direct summand of $G$, and $G/G_0$ is in fact \'{e}tale, and isomorphic to copies of the \'{e}tale part of
$\G_m$.

\section{Further observations and examples}\label{sfurth}

An instructive example of a $p$-divisible group $G=G^\prime$ is a non-trivial \'{e}tale extension of the 
multiplicative group. For a commutative $\Z[q,q^{-1}]$-algebra $R$, its coordinate ring is the inverse limit over $n$ of
\beg{etate1}{\prod_{i=0}^{p^n-1}R[t]/(t^{p^n}-q^i).
}
If we denote by $f_i$ the idempotent corresponding to the $i$th factor, the group law is given by 
\beg{etate111}{\psi(f_it)=\sum_{j=0}^{i}f_jt\otimes f_{i-j}t +\sum_{j=i+1}^{p^n-1}f_jt\otimes f_{p^n-i+j}t/q.}
As explained in Katz-Mazur \cite{katzmazur}, Chapter 8, this group can be used
to describe
the $p^n$ torsion on the Tate curve, and thus is related to our earlier discussion of elliptic cohomology.

The ring $R$ is not required to be complete, but the discussion of Chapter \ref{spdiv} applies, since we have the
exact sequence  \rref{espdiv}. In particular, the coordinate ring of the group ${}_{p^r}G$ is \rref{etate1}
with $r=n$ (where to prevent confusion, we replace $t$ by another variable $s$):
\beg{etate2}{\mathcal{O}_{{}_{p^r}G}=\prod_{i=0}^{p^r-1}R[s]/(s^{p^r}-q^i).}
We shall also denote by $e_i$ the idempotents in \rref{etate2} corresponding to $f_i$. 
On the other hand, the coordinate ring of
$G^{\prime\prime}$ is the inverse limit over $n$ of
$$\prod_{i=0}^{p^r-1}R[t]/(t^{p^n}-q^{ip^{n-r}}).$$
What is the equivariant formal group law associated with this group? First, we need to choose the parameter 
$z$. We can choose
\beg{etate3}{z=f_0(t-1)+\sum_{i=1}^{p^r-1}f_i\alpha_i
}
for some unit $\alpha_i\in (R[s]/(s^{p^r}-q^i))^\times$. (Here we should think $s=t^{p^{n-r}}$.)
We will then have
\beg{etate4}{u_j=(s^j-1)f_0+\sum_{i=1}^{p^r-1} [j]\alpha_{i}f_{ij}
}
for $j=1,\dots,p^r-1$, where the multiplication $ij$ in the subscript 
is performed in $\Z/p^r$ and $[j]$ denotes the $j$-multiple
in ${}_{p^r}G$.

Note that these $\Z/p^r$-equivariant formal group laws are not multiplicative in the sense of \cite{cgk},
since that implies $[p^r]u_i=0$, which we see from \rref{etate4} does not hold, even though it is true
non-equivariantly (or upon completing at $u=u_1$).

For illustration, let us compute explicitly the equivariant formal group law coproduct in the case $p=2$, $r=1$. 
Let us choose $\alpha_1=-1$, i.e.
\beg{etate20}{x=z=f_0t-1.}
Now \rref{etate111} implies 
\beg{etate21}{x_\alpha=e_0f_0st+\frac{e_1f_1st}{q}-1.}
Putting $A_0=\Z[q,q^{-1}]$, we then have
\beg{etate22}{\begin{array}{l}\Xi=(\mathcal{O}_{\Gamma\times G})^\wedge_{(xx_\alpha)}=\\[1ex]
e_0f_0A_0[s]/(s^2-1)[[x]]\prod e_1f_0A_0[s]/(s^2-q)[[x]]\\[1ex]
\prod e_1f_1A_0[s]/(s^2-q)[[x_\alpha]]
\end{array}}
(to make the formulas less cluttered, we will omit all hats). Now by \rref{etate20},
we have
\beg{etate23}{\psi(x)=f_0t\otimes f_0t+\frac{f_1t\otimes f_1t}{q}-1\otimes 1.
}
Comparing with \rref{etate22}, on the $f_0\otimes f_0$ coordinates, 
we just get the formula for the multiplicative FGL, i.e.
\beg{etate24}{(x+1)\otimes (x+1)-1\otimes 1.}
Further, on the $f_0\otimes f_1$ or $f_1\otimes f_0$ coordinates, only the last term \rref{etate23} 
contributes, so \rref{etate24} continues to be correct. On the $f_1\otimes f_1=e_1f_1\otimes e_1f_1$-coordinate,
(the equality comes from \rref{etate22}), we get the term
\beg{etate25}{\frac{t\otimes t}{q}-1\otimes 1,}
while in the $e_1f_1$-coordinate, $x_\alpha$ is equal to
\beg{etate26}{\frac{st}{q}-1.}
Using also the equality $s^2=q$, on the $e_1\otimes e_1$-coordinate we get that \rref{etate25} is equal to
$$x_\alpha\otimes x_\alpha+x_\alpha\otimes 1+1\otimes x_\alpha.$$
Thus, the equivariant formal group law multiplication on $\Xi$ is given by
$$\psi(x)=(x+1)\otimes (x+1)-(1\otimes 1)+e_1(x_\alpha+1)\otimes e_1(x_\alpha+1)$$
with the first two summands coinciding with the equivariant multiplicative FGL, the 
last summand being the correction.

\vspace{3mm}
We can turn the equivariant formal group laws coming from the above group $G$ 
into $\Z/p^r$-equivariant spectra using the method of
Section \ref{szpr}. (Note that we need an additional element to serve as the Bott element.) 
The fixed point spectrum of this construction, however, is a module over the non-equivariant
K-theory (or connective k-theory) spectrum, depending on whether we invert the Bott element or not. 
Thus, the fixed point spectrum carries no chromatic height $>1$ information, even though
the group $G$ is height $2$ according to the classification of $p$-divisible groups.

\vspace{3mm}

The example brings out the fairly major role the choice of parameter \rref{etate3} appears to have
in the construction. Thus, we can ask to what extent the resulting equivariant spectra depend on
the choice of parameter. To this end, we must recall how reparametrizations work in equivariant formal 
group laws and complex-oriented
equivariant spectra \cite{cgk,cgk1}. Recall first that non-equivariantly, the formal
group law of a complex-oriented ring spectrum is uniquely determined up to {\em strict isomorphism},
which means that the two formal group laws are related by a reparametrization of the form
\beg{etate5}{x+\sum_{i>1} a_ix^i.
}
Two complex orientations of the spectrum are strictly isomorphic.
The coefficient of the linear term is fixed by the requirement that the Thom class of the
universal complex line bundle on $\C P^\infty$ restrict to the element $1\in \widetilde{E}^2\C P^1=E_0$.

$G$-equivariantly (where $G$ is an abelian compact Lie group), this remains true when interpreted properly.
The complex orientation is a Thom class 
$$\tau\in \widetilde{E}^2(\C P^\infty_G)^{\gamma^2_G},$$
which means that we require that it restrict to generators of the rank $1$ free $E_0$-modules
\beg{etate6}{\tau_\gamma\in\widetilde{E}^2 S^{\gamma}= E_{\gamma-2}}
for every $1$-dimensional complex $G$-representation $\gamma$. Thus, a ``strictness" requirement
is imposed once free generators $\tau_\gamma$ in \rref{etate6} are fixed for $\gamma\ncong 1$, which is a part of
the structure of a complex-oriented equivariant spectrum. 

These choices are not arbitrary. For example, 
for a Borel cohomology spectrum $E=F(EG_+,Z)$ for a non-equivariant complex-oriented
spectrum $Z$, they are determined. However, consider a $\Z/p^r$-equivariant spectrum $E$ which is obtained
from a $\Z/p^{r-i}$-equivariant spectrum $E^\prime$ via the adjunction of Lemma \ref{geomr}.
Then, by definition, the stable homotopy class
$$s_\gamma:S^0\rightarrow S^{\gamma}$$
given by inclusion of fixed points is invertible whenever $\gamma|_{\Z/p^{r-i}}$ is trivial.
The corresponding Euler class is given by  
$$u_\gamma=\tau_\gamma\circ s_\gamma,$$
and thus, is invertible.  

Since the $\Z/p^r$-equivariant spectrum associated with the group $\mathcal{O}_{{}_{p^r}G}$ of \rref{etate1}
is, by construction, a wedge sum of spectra of the type just discussed, we see that a choice of a parameter $z$
determines a compatible choice of complex stability isomorphisms \rref{etate6} 
(not necessarily including every such choice). Similar comments apply
to the spectra constructed in Section \ref{szpr}. In this sense, the $\Z/p^r$-equivariant complex-oriented 
spectrum
constructed is, in fact, intrinsic to the $p$-divisible group, and does not depend on the choice of parameter.

\vspace{3mm}

All the equivariant spectra we have constructed so far are direct
(i.e. wedge) sums of terms corresponding to the higher diagonal in diagram \rref{eelliptic1000}. This is, in fact,
not surprising, since we have constructed our spectra by spectrification of the coefficients, and 
in an \'{e}tale $p^r$-torsion finite group scheme, the filtration by the $p^i$-torsion subgroups 
splits into a disjoint union. It may, however, raise the question whether all $\Z/p^r$-equivariant spectra whose
equivariant formal group laws reduce to multiplicative or Lubin-Tate non-equivariantly (or their corresponding
equivariant spectra) have such a splitting
property. 

This is certainly false. For example, it follows from our observations in Section \ref{szpr}
that if $E_*$ is a ring obtained from $MU_*$ by killing a regular sequence and localizing and 
$$[p]x:E^*[[x]]\rightarrow E^*[[x]]$$
is injective, then the regular sequence and localization lifts to a regular sequence in $MU_*^{\Z/p^r}$, thereby giving
rise to a $MU_*^{\Z/p^r}$-module $E_*^{\Z/p^r}$ and, when
$E_*$ has no torsion, to a corresponding $\Z/p^r$-equivariant complex-oriented spectrum $E_{\Z/p^r}$.
It further follows from the results of \cite{hu} that
in such a case, the scheme $Spec(E_*^{\Z/p^r})$ is connected. These rings are, however, 
non-Noetherian.

\vspace{3mm}

Nevertheless, there are examples which are Noetherian, although their construction
involves the full force of the results of \cite{hu}, which are too technical
to reproduce here. For this reason, we restrict attention here to the group $\Z/2$, 
in which case, we have the ``classical" result
\beg{ez22cob}{\begin{array}{l}MU^{\Z/2}_*=\\[1ex]
MU_*[u, q_j, b_{ij}\mid i\geq 1, 
j\geq 0]/(q_0,q_i-c_i-uq_{i+1}, b_{ij}-a_{ij}-ub_{i,j+1})
\end{array}}
where 
$$x+_Fy=\sum_{i,j} a_{ij}x^iy^j,\; [2]x=\sum_{j}c_jx^j.$$
This is a minor modification of the presentation of Strickland \cite{strick}, and is
readily proved by the same
methods (the advantage being that the variables $q_j$ are separated from the $b_{ij}$). In Borel cohomology, we
have
\beg{ez22cob1}{q_j\mapsto \sum_{k\geq j} c_ku^{k-j}}
\beg{ez22cob1a}{
b_{ij}\mapsto \sum_{k\geq j} a_{ik}x^iu^{k-j}
}
With these conventions, we see that in the multiplicative formal group law on 
$$\Z[u]/(u(u+2)),$$ 
we have
\beg{ez22cob2}{a_{11}=1, \; a_{ij}=0 \text{ for $i>1$ or $j>1$}
}
\beg{ez22cob3}{b_{11}\mapsto 1, b_{ij}\mapsto 0 \text{ for $i>1$ or $j>1$},
}
\beg{ez22cob4}{q_1\mapsto u+2,\; q_2\mapsto 1,\; q_j\mapsto 0 \text{ for $j>2$}.
}
Now we can construct a $\Z/2$-equivariant formal group law on the Noetherian ring
\beg{ez22cob5}{A=\Z[u,w]/(u(u+2)(1-uw)),
}
whose spectrum (in the sense of algebraic geometry)
is connected. To this end, we keep the relations \rref{ez22cob2}, \rref{ez22cob3},
and send 
\beg{ez22cob6}{q_k\mapsto -u(u+2)w^{k} \text{ for $k\geq 3$}
}
(then $q_1,q_2$ are determined). The $\Z/2$-equivariant formal group law thus constructed is
not multiplicative in the sense of \cite{cgk}, but its underlying non-equivariant formal group 
law is multiplicative. 

This equivariant formal group law can be figured out explicitly as follows:
Let us first discuss its structure when we change scalars to the ring $u^{-1}A$. To this end,
let us recall that a $\Z/2$-equivariant formal group law on a ring $B$ where the Euler class is
a unit is on a ring of the form
\beg{ez22cob50}{B[[z]]\prod B[[z]].}
Denote the idempotents corresponding to the two factors \rref{ez22cob50} by $\iota_0,\iota_1$.
Then the formal comultiplication is given by
\beg{ez22cob51}{\begin{array}{l}
\psi(\iota_0z)=(\iota_0\otimes\iota_0)(z_1+_Fz_2)+(\iota_1\otimes\iota_1)(z_1+_Fz_2)\\[1ex]
\psi(\iota_1z)=(\iota_0\otimes\iota_1)(z_1+_Fz_2)+(\iota_1\otimes\iota_0)(z_1+_Fz_2)
\end{array}
}
where $z_1=z\otimes 1$, $z_2=1\otimes z$. The choice of parameter is given by 
\beg{ez22cob52}{\begin{array}{l}
x=(z,u+\sum_{i\geq 1} b_iz^i),\\[1ex]
x_\alpha=(u+\sum_{i\geq 1} b_iz^i,z).
\end{array}
}
The representing object is, in fact, $MU_*[u,u^{-1},b_i\mid i\geq 1]$. In our case, we have
$b_1=1+u$, $b_i=0$ for $i>1$, so \rref{ez22cob52} turns into
\beg{ez22cob53}{\begin{array}{l}
x=(z,u+(1+u)z),\\[1ex]
x_\alpha=(u+(1+u)z,z).
\end{array}
}
Since the $q_i$s are not a part of this picture, one may ask how could the equivariant formal
group fail to be multiplicative? Computing this out explicitly using \rref{ez22cob53}, however, reveals that
\beg{ez22cob54}{x\otimes x+x\otimes 1+1\otimes x -\psi(x)=u(u+2)\iota_1(z+1)\otimes \iota_1(z+1).
}
Thus, we see that for any ring with the Euler class inverted, the $\Z/2$-equivariant formal group law
is multiplicative if and only if $u(u+2)=0$, which is not our case. (This is, of course, also proved in 
\cite{cgk}.) The correction is given by the
right hand side of \rref{ez22cob54}.

To evaluate that term, let us first note from \rref{ez22cob53} that
\beg{ez22cob55}{xx_\alpha=z^2(1+u)+zu.
}
When $u$ is a unit, setting $t=xx_\alpha$, we have in fact an inverse series
\beg{ez22cob56}{z=\phi(t,u).
}
Schematically, we can write
$$\phi(t,u)=\frac{-u+\sqrt{u^2+4(1+u)t}}{2(1+u)},
$$
and thus the coefficients can be expressed by the Newton formula. 

Now by \rref{ez22cob53}, we have 
\beg{ez22cob57}{\iota_1(z+1)=(0,z+1)=u^{-1}(x-\phi(xx_\alpha,u)).
}
This can be used to express the correction term \rref{ez22cob54}. However, because of our factor
$u(u+2)$, we can also use the identity $w=u^{-1}$. In fact, making this substitution, \rref{ez22cob55}
becomes
$$wxx_\alpha=z^2(w+1)+z,$$
from which we see that the series $\phi(t,w^{-1})$ is integral in $t,w$, and hence, so is the series
$$\rho(x,x_\alpha,w)=w(x-\phi(xx_\alpha,w^{-1})).$$
Therefore, the identity
\beg{z2cob58}{
\psi(x)=x\otimes x+1\otimes x+x\otimes 1 -u(u+2)\iota_1\rho(x,x_\alpha,w)\otimes 
\iota_1\rho(x,x_\alpha,w)
}
can be written over the ring $A$, and hence is then valid by inverting and completing at $u$ (which 
eliminates $\iota_1$).

\vspace{3mm}

Additionally, (omitting the Bott element from the notation), it is not difficult to construct
a complex oriented $\Z/2$-equivariant spectrum with this formal group law. Multiplying the relation by 
the Bott element, the element $w$ can be in degree $0$, and thus can be attached so that
$MU_{\Z/2}[u]$ is a $\Z/2$-equivariant $E_\infty$ ring. Then the relation 
$u(u+2)(1-uw)$ together with the non-equivariant complex cobordism polynomial generators $x_i$, $i>1$
form a regular sequence ($x_1$ is not killed, going to the Bott element). The relations
regarding $b_{ij}$, $q_j$ cannot be arranged into a regular sequence, but this can be
circumvented by the following trick: Selecting $k\geq 3$, the elements
$$q_k+u(u+2), b_{1k},b_{2k},\dots$$
(together with the elements mentioned above) do form a regular sequence,
and one can take the homotopy colimit (telescope) of the resulting spectra with $k\r\infty$.

\vspace{3mm}
For Lubin-Tate formal group laws on non-equivariant complete rings $E_{n*}$, the Borel cohomology equivariant
Lubin-Tate law on
$$E_{n*}[[u]]/([2]u)$$
can be similarly deformed to a $\Z/2$-equivariant formal group law on
\beg{ez22cob10}{E_{n*}[[u]][w]/(([2]u)(1-uw))
}
by keeping the relations \rref{ez22cob1a}, but deforming \rref{ez22cob1} to 
$$ q_j\mapsto \sum_{k\geq j} c_ku^{k-j} -w^j[2]u.$$
We can also construct $\Z/2$-equivariant complex oriented spectra with this coefficient ring
(neglecting the periodicity element from the notation) by the same method as in the multiplicative case.

\vspace{5mm}

\section{Appendix} \label{sapp}

In this section, we clarify the relationship between the isotropy separation diagram \rref{ezpreq}, which we
used to construct equivariant elliptic and Barsotti-Tate cohomology, with the calculation of $(MU_{\Z/p^r})_*$
given in \cite{ak}. In \cite{ak}, there is a ``staircase-shaped diagram" 
\beg{ezpreq100}{\resizebox{0.9\hsize}{!}{\diagram
& F(E\mathcal{F}[\Z/p^{i+1}]_+,\widetilde{E\mathcal{F}[\Z/p^i]}\wedge X)\dto\rto &\dots\\
\dots\rto & F(E\mathcal{F}[\Z/p^{i+1}]_+,\widetilde{E\mathcal{F}[\Z/p^i]}\wedge 
F(E\mathcal{F}[\Z/p^i]_+,\widetilde{E\mathcal{F}[\Z/p^{i-1}]}\wedge X)).&
\enddiagram}
}
However, it should be pointed out that a general $\Z/p^r$-equivariant spectrum $X$ is {\em not} the
homotopy limit of the diagram \rref{ezpreq100}. Rather, it is a homotopy limit of a larger ``cube-like"
diagram of which \rref{ezpreq100} is a part.

What happens for cobordism is that this homotopy limit has no higher derived functors, and therefore to calculate
the coefficients of $\Z/p^r$-equivariant cobordism, we can just take the limit of the coefficients of the spectra
in the cube-like diagram. This limit is then easily seen to be equal to the limit of the coefficients of the
spectra \rref{ezpreq100} for $X=MU_{\Z/p^r}$, i.e.
\beg{ezpreq1011}{\resizebox{0.8\hsize}{!}{\diagram\protect
& F(E\mathcal{F}[\Z/p^{i+1}]_+,\widetilde{E\mathcal{F}[\Z/p^i]}\wedge MU_{\Z/p^r})_{*}\rto\dto&\dots \\
\protect\dots\rto & 
\protect{\begin{array}{l}F(E\mathcal{F}[\Z/p^{i+1}]_+,\widetilde{E\mathcal{F}[\Z/p^i]}\wedge \protect
\\F(E\mathcal{F}[\Z/p^i]_+,\protect\widetilde{E\mathcal{F}[\Z/p^{i-1}]}\wedge\protect MU_{\Z/p^r}))_{*}.
\end{array}}&
\enddiagram}
}
In this case, the coeffients of the top term of \rref{ezpreq100} are
\beg{ezpreq1012}{MU_*[u_\alpha,u_{\alpha}^{-1},b^\alpha_j\mid
j\in \N, 1\leq \alpha<p^{i}]
[[u_{p^i}]]\\[1ex]
/[p^{r-i}]u_{p^i}.}
while the coefficients of the bottom term 
are
\beg{ezpreq1022}{(MU_*[u_\alpha,u_{\alpha}^{-1}, b_i^{\alpha}\mid 1\leq \alpha\leq p^{i-1}-1][[u_{p^{i-1}}]]
[u_{p^{i-1}}^{-1}]/
[p^{r-i+1}]u_{p^{i-1}})^\wedge_{([p]u_{p^{i-1}})}.}
(The completion symbol is actually missing in \cite{ak}.)

\vspace{3mm}

For any $\Z/p^r$-equivariant spectrum $X$, there is an obvious comparison map from diagram \rref{ezpreq}
to diagram \rref{ezpreq100}. In the case of $X=MU_{\Z/p^r}$, on coefficients, the map is injective.
It is proved in \cite{hu} that for connectivity reasons, this map on coefficients
is an isomorphism for $i\leq 1$. It is also an isomorphism on \rref{ezpreq1012} for $i=r$
(i.e., the
top term \rref{ezpreq1011}). Thus, 
for $r\leq 2$, the comparison map from Diagram \rref{ezpreq} to Diagram \rref{ezpreq100}
is an equivalence on fixed points. 

As it turns out, in all other cases, i.e. on \rref{ezpreq1012} with $2\leq i<r$ and on 
\rref{ezpreq1022} with $2\leq i\leq r$, the comparison map fails to be an isomorphism. 
The reason is that in those cases, \rref{ezpreq1012}, \rref{ezpreq1022} actually contain infinite
series in $u_{p^i}$ with unbounded $u_{p^{i-1}}$-denominators, which can be dimensionally compensated
by positive powers of $u_{p^j}$ with $j<i-1$. For $i\leq 1$, there are no such compensating terms.
A detailed proof of this is given in \cite{hu}.

Therefore,
in particular, while on coefficients, the diagram \rref{ezpreq} has no higher derived limits,
there are, in fact, non-zero higher derived limits of the diagram \rref{ezpreq1011} for $r>2$,
although its limit is still $(MU_{\Z/p^r})_*$. From this point of view, the diagram \rref{ezpreq}
is ``minimal" on coefficients, and therefore it is a simpler model of isotropy separation for $\Z/p^r$
than the diagram of \cite{ak}.

\vspace{5mm}

Po Hu

Department of Mathematics,
Wayne State University

1150, Faculty/Administration Bldg 

656 W Kirby St, Detroit, MI 48202

U.S.A.

pohu@wayne.com

\vspace{3mm}

Igor Kriz

Department of Mathematics,
University of Michigan

530 Church Street

Ann Arbor, MI 48109-1043

U.S.A.

ikriz@umich.edu

\vspace{3mm}

Petr Somberg

Mathematical Institute of MFF UK

Sokolovsk\'{a} 83, 18000 Praha 8

Czech Republic

somberg@karlin.mff.cuni.cz

\end{document}